\documentclass{article}
\usepackage{float}
\usepackage{graphicx,epstopdf}
\usepackage{bbm}
\usepackage{color}
\usepackage{amsfonts}
\usepackage{mathrsfs}
\usepackage{amsmath}
\usepackage{dsfont}
\usepackage{amssymb}
\usepackage{csquotes}
\usepackage{theorem}
\usepackage{graphicx}
\usepackage[dvips]{epsfig}
\usepackage{latexsym}
\usepackage{exscale}
\usepackage[latin1]{inputenc}
\newtheorem{theorem}{Theorem}

\newtheorem{definition}[theorem]{Definition}
\newtheorem{example}[theorem]{Example}

\newtheorem{lemma}[theorem]{Lemma}

\newtheorem{remark}[theorem]{Remark}

\newenvironment{proof}[1][Proof]{\noindent\textbf{#1.} }{\ \rule{0.5em}{0.5em}}
\newcommand\blfootnote[1]{%
  \begingroup
  \renewcommand\thefootnote{}\footnote{#1}%
  \addtocounter{footnote}{-1}%
  \endgroup
}

\begin{document}
\title{Subspace Condition for Bernstein's Lethargy Theorem}
\date{\vspace{-5ex}}
\maketitle

Asuman G\"{u}ven AKSOY\\
corresponding author, Claremont McKenna College\\
Department of Mathematics, Claremont, CA  91711, USA \\
E-mail: aaksoy@cmc.edu \\

Monairah AL-ANSARI\\
Claremont Graduate University\\
Institute of Mathematical Sciences, Claremont, CA  91711, USA\\
E-mail: monairah.alansari@cgu.edu \\

Caleb CASE\\
Claremont McKenna College\\
Department of Mathematics, Claremont, CA  91711, USA \\
E-mail: caleb.case16@cmc.edu \\

Qidi PENG\\
Claremont Graduate University\\
Institute of Mathematical Sciences, Claremont, CA  91711, USA \\
E-mail: qidi.peng@cgu.edu \\

\vspace{.15 cm}

\mbox{~~~}\\
\mbox{~~~}\\
\noindent\textbf{Abstract:}  In this paper, we consider a condition on subspaces in order to improve bounds given in the Bernstein's Lethargy Theorem (BLT) for Banach spaces.  Let $d_1 \geq d_2 \geq \dots d_n \geq \dots > 0$ be an infinite sequence of numbers  converging to $0$, and let $Y_1 \subset Y_2 \subset \dots\subset Y_n \subset \dots \subset X$ be a sequence of  closed nested subspaces in a Banach space $X$  with the property that  $\overline{Y}_{n}\subset Y_{n+1}$  for all  $n\ge1$. We prove that  for any $c \in (0,1]$, there exists an element $x_c \in X$ such that
$$ c d_n \leq \rho(x_c, Y_n) \leq \min (4, \tilde{a}) c\, d_n.
$$
Here, $\rho(x, Y_n)= \inf \{ ||x-y||: \,\,y\in Y_n\}$,  $$\tilde{a} =\sup_{i\ge1}\sup_{\left \{ q_{i}  \right \}}\left \{  a_{n_{i+1}-1}^{-3}\right \}$$  where the sequence $\{a_n\}$  is defined as: for all $ n \geq 1 $,
$$
a_n = \inf_{l \geq n} \,  \inf_{q \in \langle q_l, q_{l+1},\dots \rangle} \frac{\rho(q,Y_l)}{||q||}
$$
in which each point $q_n$ is taken from  $Y_{n+1} \setminus Y_{n}$, and satisfies  $\inf\limits_{n\ge1} a_n  > 0$. The sequence  $\{n_i\}_{i\ge1}$ is given by
$$
n_1=1;~n_{i+1}= \min \left \{ n\ge1 : \frac{d_n}{{a_{n}^{2}}} \leq d_{n_{i}}\right \},~i\geq 1.
$$
\noindent\blfootnote{{\bf Mathematics Subject Classification (2000):}
41A25, 41A50, 41A65. \\\noindent\vskip1mm {\bf ~~Key words: } Best  Approximation, Bernstein's Lethargy Theorem, Banach Spaces.}

\section{Introduction}
Bernstein's Lethargy Theorem \cite{Bernstein} involves finding approximations of an element in a space $X$ when those approximations are limited to some sequence of subspaces. Before we can compare approximations, we need a function to determine how close an approximation is to the desired target. In the following we define a distance function which we call  the $\rho$-function:
\begin{definition}
	Let $(X,\|\cdot\|)$ be a Banach space, and let $S$ be a subspace of $X$. Then, for any point $x \in X$, we can define the distance from $x$ to $S$ as $$\mbox{dist}(x,S) = \rho(x,S) = \inf_{y\in S} ||x-y||.$$
\end{definition}
If $Y_1\subset Y_2 \subset \dots$ is a sequence of strictly embedded linear subspaces of $X$, then for each $x\in X$, there exists a non-increasing sequence of best approximation errors $$\rho(x, Y_1) \geq \rho(x, Y_2) \geq \dots.$$ The general objective is to characterize these sequences of best approximation errors. For example, one can ask if  it is  true that for any non-increasing sequence $\{d_n\}$ with $\lim\limits_{n \to \infty} d_n =0$, there exists an element $x\in X$ such that $$\rho(x, Y_n) =d_n \quad \mbox{for all}\quad  n=1,2,\dots. $$   Bernstein \cite{Bernstein} proved that in the case $X= C[a,b]$ and $ Y_n=P_n  $,  the space of polynomials of degree at most $n$, \textbf{any} sequence converging to zero is a sequence of best approximations.
This theorem is sometimes referred  to as Bernstein's Lethargy Theorem (BLT) and it has been applied to the theory of quasi analytic functions in several complex variables \cite{Ple2} and used in the constructive theory of functions \cite{Sin}. Note that the density of polynomials in $C[0,1]$ (Weierstrass Approximation Theorem \cite{Cheney}) implies that $\displaystyle \lim_{n \to \infty} \rho(f, P_n) = 0.$ However, Weierstrass Approximation Theorem gives no information about the speed of convergence for $\rho(f, Y_n)$.  Following the proof of Bernstein, Timan \cite{Timan} extended his result to an arbitrary system of strictly embedded \textit{finite-dimensional} subspaces $Y_n$.  Later Shapiro \cite{Sha}, replacing $C[0,1]$  with an arbitrary infinite-dimensional Banach space $(X, \|\cdot\|)$ and the sequence of $n$-dimensional subspaces of polynomials of degree $\leq n$ by a sequence $\{Y_n\}$ where $Y_1 \subset Y_2 \subset \cdots$ are strictly embedded \textit{closed subspaces} of $X$, showed that in this setting, for each null sequence $\{d_n\}$ of non-negative numbers, there is a vector $x\in X$ such that
$$ \rho(x, Y_n) \neq O(d_n),~\mbox{as}~n\to\infty.$$
Thus, there is no $M>0$ such that $\rho(x, Y_n) \leq M d_n$ for all $n$. In other words, $\rho(x, Y_n)$ can decay arbitrarily slowly. This result  was strengthened by Tyuriemskih \cite{Tyu} who established that the sequence of best approximations may converge to zero at an arbitrary slow rate:  for any expanding sequence $\{Y_n\}$ of subspaces and for any sequence $\{d_n\}$ of positive numbers converging to zero, he constructed an element $x\in X$  such that $\displaystyle\lim_{n \rightarrow \infty} \rho(x, Y_n) =0$ and $\rho(x, Y_n) \geq d_n$ for all $n.$ For a generalization of Shapiro's Theorem we refer the reader to \cite{Al-Oi}. For an application of Tyuriemskih's Theorem to convergence of sequence of bounded linear operators, consult \cite{De-Hu}. For other versions of Bernstein's Lethargy Theorem, see \cite{Ak-Al,Ak-Le,Ak-Le2,Al-To,Lew1,Pli}.\\

We now consider the following well-known  Bernstein's Lethargy Theorem \cite{Bernstein}, stated for the case of  finite-dimensional subspaces of a Banach space $X$.

\begin{theorem}[Lethargy] \label{Bernstein}
	Given a Banach space $X$ and a series of nested finite-dimensional subspaces $Y_1 \subset Y_2 \subset \dots \subset X$. If $\{d_k\}_{k\ge1}$ is a monotone decreasing sequence converging to $0$, then there exists a point $x \in X$ such that $\rho(x,Y_k) = d_k$ for all $k \ge1$.
\end{theorem}

The above theorem can be extended to infinite-dimensional subspaces, by considering some  extra conditions. Borodin \cite{Borodin} has provided two sets of conditions. One condition is on the sequence $\{d_n\}$  and the other on both the subspaces $\{Y_n\}$ and the sequence $\{d_n\}$. In both cases, he proves the existence of an element $x\in X$ with  $\rho(x,Y_k) = d_k,\,\, k \geq 1.$ These two results are explicitly presented as follows.
\begin{theorem}(see \cite{Borodin})\label{Borodin1}
\label{Borodin:first}
Let $X$ be an arbitrary infinite-dimensional Banach space, $Y_1\subset Y_2\subset\ldots$ be an arbitrary system of strictly embedded subspaces in $X$, and the number sequence $\{d_n\}$ be such that
\begin{equation}
\label{condition:d}
d_n>\sum_{k=n+1}^{\infty}d_k
\end{equation}
for every positive integer $n\ge n_0$ for which $d_n>0$. Then there exists an element $x\in X$ such that $\rho(x,Y_n)=d_n$ for $n\ge1$.
\end{theorem}
\begin{theorem}(see \cite{Borodin})
\label{100}
	Let $d_0\ge d_1 \geq d_2 \geq \dots > 0$ be a non-increasing sequence converging to $0$ and $Y_1 \subset Y_2 \subset \dots \subset X$ be a system of strictly nested subspaces of  an infinite-dimensional Banach space $X$ that meets the following property: there exists a series of nonzero elements $q_n$ such that $q_n \in Y_{n + 1} \setminus Y_n$, and  the following inequality
\begin{equation}
	\label{Borodin2}
\|q\| \leq \displaystyle \frac{d_{k-1}}{d_{k}}\rho(q,Y_k)
\end{equation}
holds for all $k \in \mathbb N$ and any nonzero element $q$ in the linear span $\langle q_k, q_{k + 1},\dots \rangle$. Then there is some element $x$ in the closed linear span $\overline{\langle q_1, q_2,\dots \rangle}$ satisfying $$\rho(x,Y_n) = d_n ~~for ~all ~n \ge1.$$
\end{theorem}
 Recently Konyagin \cite{Konyagin} showed that under the same assumptions in Theorem \ref{Borodin:first}, except that the sequence $\{d_n\}$ can go to $0$ with arbitrary rate, there is $x\in X$ such that
 \begin{equation}
 \label{Konyagin}
 d_n\le \rho(x,Y_n)\le 8d_n,~\mbox{for $n\ge1$}.
 \end{equation}
 The proof is based on Theorem \ref{Borodin:first}. Note that the statements in Theorem \ref{100} are similar to that of Theorem \ref{Borodin:first}. We can now  adapt  the idea of the proof of Konyagin's \cite{Konyagin} with the Borodin's theorem \cite{Borodin} to improve the bounds of $\rho(x,Y_n)$ in (\ref{Konyagin}).

 In Konyagin's paper \cite{Konyagin}, it is assumed that $Y_n$ are closed
and strictly increasing. In Borodin's paper, this is not specified, but from the proof of his theorem
it is clear that his proof  works only under the assumption that $\overline{Y_n} $ is strictly included
in $Y_{n+1}$.   Necessity of this assumption on subspaces is illustrated by the following:
\begin{example}
Let $X = L^{\infty}[0,1]$ and consider ${C}[0,1] \subset L^{\infty} [0,1]$. Define the subspaces  of $X$ as follows:

\begin{itemize}
\item  $ Y_1 =  \mbox{all polynomials}$
\item  $Y_2 =$ span$ [Y_1 \cup f_1],$
where $ f_1 \in {C}[0,1] \setminus Y_1$,
\item  $ Y_{n+1} = $span$[Y_n \cup f_{n}]$
where $f_{n} \in {C}[0,1] \setminus Y_n$.
\end{itemize}
Observe that by Weierstrass Theorem, $\overline {Y_n} = {C}[0,1] $ for any $n \ge1$.
Take any  $ f \in L^{\infty}[0,1]$ and consider the following cases:
\begin{enumerate}
  \item If  $ f \in {C}[0,1]$ , then

$$\mbox{dist} (f,Y_n) = \mbox{dist}(f,{C}[0,1]) = 0~ \mbox{for any}~ n\ge1.$$
  \item If $ f \in L^{\infty}[0,1] \setminus {C}[0,1] $, then
$$\mbox{dist}(f, Y_n) = \mbox{dist}(f, {C}[0,1]) = d > 0~ \mbox{ (independent of $n$)}.$$
\end{enumerate}
Hence in this case BLT does not hold. (Note that in the above, we used the fact that $\mbox{dist}(f,Y) = \mbox{dist}(f,\overline{Y})$.)
\end{example}
Note that Borodin's condition on sequence $\{d_n\}$, namely
$d_n>\sum\limits_{k=n+1}^{\infty}d_k$
 is not satisfied when $d_n= \displaystyle\frac{1}{2^n}$, however it is satisfied when $d_n =\displaystyle\frac{1}{(2+\epsilon)^n} $ for $\epsilon > 0$.  Thus, it is natural to ask
 whether the condition (\ref{condition:d}) is necessary for the results in Theorem \ref{Borodin:first} to hold?

 In \cite {Ak-Pe} , it is shown that  by weakening the condition (\ref{condition:d}) in Theorem \ref{Borodin1} above yields to an improvement of the bounds in the inequality  (\ref{Konyagin}) in Konyagin's theorem. In this paper, we take a different approach. We concentrate on the Borodin's second condition on subspaces, namely on the inequality  (\ref{Borodin2}) of Theorem \ref{100}  above and obtain better bounds for the inequalities in (\ref{Konyagin}).  The statements in Theorem \ref{100} are similar to that of Theorem \ref{Borodin:first}, thus, we can now  adapt  the idea of the proof of Konyagin's \cite{Konyagin} with the Borodin's theorem \cite{Borodin} to improve the bounds of $\rho(x,Y_n)$ in (\ref{Konyagin}).

\section{Main Result}
Let $X$ be an arbitrary infinite-dimensional Banach space. Given $Y_1 \subset Y_2 \subset \dots \subset X$, an arbitrary system of strictly embedded closed subspaces and $d_1 \geq d_2 \geq \dots \geq 0$, a non-increasing sequence converging to 0. The goal of this paper is to improve Konyagin's result (\ref{Konyagin}) under conditions on subspaces $\{Y_n\}$ and $\{d_n\}$. It is worth noting that if $\{Y_n\}$ and $\{d_n\}$ are finite sequences, we have the best approximations of the sequence $d_n$ in terms of the distances $\rho(x,Y_n)$, i.e., the following lemma holds:
\begin{lemma}
\label{1}
		Let $d_1 > d_2 > \dots > d_n >0$ be a finite decreasing sequence and $Y_1 \subset Y_2 \subset \dots \subset Y_n$ be a system of strictly nested closed subspaces of Banach space $X$. Then for any $c\in(0,1]$, there exists  an element $x_c \in X$ such that   $\rho(x_c,Y_k) = cd_k$, for $k = 1,\dots n$.
	\end{lemma}
	\begin{proof}
		First, from \cite{Borodin} and \cite{Timan}, we see Lemma \ref{1} is true for $c=1$. Next for any $c\in(0,1]$, let $\tilde d_k=cd_k$. It is easy to see the sequence of numbers $\{\tilde d_k\}$ satisfies Lemma \ref{1}, therefore there exists  an element $x_c \in X$ such that   $\rho(x,Y_k) = \tilde d_k$, for $k = 1,\dots, n$.
	\end{proof}

Now we consider the case when $\{Y_n\}$ and $\{d_n\}$ are infinite sequences and state our main result.


\begin{theorem}
\label{101}
	Let $X$ be an arbitrary infinite-dimensional Banach space, and let $Y_1 \subset Y_2 \subset \dots \subset X$ be an arbitrary system of strictly embedded closed linear subspaces. Let $ \{d_n\}$ be a non-increasing sequence of real numbers converging to $0$. Assume that, for any sequence of elements $q_i$ such that $q_n\in Y_{n+1}\setminus Y_n$ for all $n$, we have that
  \begin{equation}
  \label{boundan}
  \inf\limits_{n\ge1} a_n  > 0,
  \end{equation}
  where for each $n\ge1$, $a_n $ is defined by
 $$ a_n = \inf_{l \geq n} \,  \inf_{q \in \langle q_l, q_{l+1},\dots \rangle} \frac{\rho(q,Y_l)}{||q||}$$
for these elements $q_i$.  Then for any constant $c\in(0,1]$, there exists an element $x_{c}\in X$ such that
\begin{equation} \label{main}
cd_n \leq \rho(x_{c}, Y_n) \leq \min(4,\tilde{a})cd_n,
\end{equation}
where
$$\tilde{a} =\sup_{i\ge1}\sup_{\left \{ q_{i}  \right \}}\left \{  a_{n_{i+1}-1}^{-3}\right \}$$

and  $n_i$ satisfies
\begin{eqnarray}
\label{ni}
&&n_1=1;\nonumber\\
&&{n_{i+1}}= \min \left \{ n\ge1 : \frac{d_n}{{a_{n}^{2}}} \leq d_{n_{i}}\right \},~i\geq 1.
\end{eqnarray}
\end{theorem}
\begin{proof}
	If $d_{n}=0$ for some $n$, then Theorem \ref{101} holds by using Lemma 1 in \cite{Borodin} and Lemma \ref{1}.	
	Thus, we will assume that $d_n>0$ for all $n\ge1$. Take the sequence $\{n_i\}$ defined in (\ref{ni}). Define a sequence of positive integers $\{j_i\}$ such that
	$$j_1=1,~ j_{i+1}=\left\{
	\begin{array}{lll}
	&j_{i}+1&~\mbox{if $n_{i+1}=n_{i}$};\\
	&j_{i}+2&~\mbox{if $n_{i+1}>n_{i}$},
	\end{array}\right.~\mbox{for $i\ge1$}.$$
	Let $$  m_{j}=\left\{
	\begin{array}{lll}
	&n_{i}&~\mbox{if $j=j_{i}$};\\
	&n_{i+1}-1&~\mbox{if $j_{i}<j<j_{i+1}$}.
	\end{array}\right.$$
	Clearly the sequence $ \left \{ m_{j} \right \}_{j\geq 1} $ is strictly increasing. Now, we define the sequences of subspaces $ \left \{ Z_{j} \right \} _{j\geq 1}$ and numbers $ \left \{ e_{j} \right \} _{j\geq 1}$ to be
\begin{eqnarray*}
 Z_{j}&=&Y_{m_{j}},\\
e_{j}&=&\left\{
	\begin{array}{lll}
	&\frac{c}{a_{m_{j+1}}}d_{n_{i}}&~\mbox{if $j=j_{i}$ for some $i$};\\
	&cd_{n_{i}}&~\mbox{if $j_{i}<j<j_{i+1}$}.
	\end{array}\right.
	\end{eqnarray*}
	Hence, for any $j\geq1$, 3 cases follow:\\
	Case 1:\\
	if $j_i<j < j_{i+1 }$ for some $i$, then $j+1= j_{i +1}$. By the definition of $e_j$, the facts that $n_{i+1}=m_{j+1}$ and $\{a_n\}_n$ is increasing, we obtain \\
	$$e_{j+1}= \frac{c}{a_{m_{j+2}}}d_{n_{i+1}}\leq \frac{c}{a_{m_{j+2}}}a_{n_{i+1}}^{2}d_{n_{i}} = \frac{c}{a_{m_{j+2}}}a_{m_{j+1}}^{2}d_{n_{i}}\leq a_{m_{j+1}}e_{j}.$$
	Case 2:\\
	if $ j= j_{i }$ for some $i$ and $j+1< j_{i +1}$, then
	$$e_{j+1}= cd_{n_i} =a_{m_{j+1}}e_{j}.$$
	Case 3:\\
	if $j= j_{i }$ for some $i$ and $j+1= j_{i +1}$, then
	$$e_{j+1}= \frac{c}{a_{m_{j+2}}}d_{n_{i+1}}\leq \frac{c}{a_{m_{j+2}}}a_{n_{i+1}}^{2}d_{n_{i}}=\frac{ a_{m_{j+1}}a_{n_{i+1}}^2}{a_{m_{j+2}}}e_{j}\leq a_{m_{j+1}}e_{j}.$$
	Thus, we conclude that $$e_{j+1}\leq a_{m_{j+1}}e_{j},~\mbox{for all}~j\geq1.$$
	Note that  for all $q\in \langle q_{m_{j+1}},q_{m_{j+1}+1},\ldots\rangle$,
	$$\frac{e_{j+1}}{e_{j}}\leq a_{m_{j+1}}\leq \frac{\rho (q,Y_{m_{j+1}})}{||q||}=\frac{\rho (q,Z_{j+1})}{||q||}.$$
	Therefore we can apply Theorem \ref{100} to the sequence $ \left \{ Z_{j} \right \} _{j\geq 1}$  of subspaces and  the sequence of numbers $ \left \{ e_{j} \right \} _{j\geq 1}$, to obtain the existence of an element $x_{c}' \in\overline{< q_{1},q_{2},\ldots> }$ such that $$\rho (x_{c}',Z_{j})=e_{j},~\mbox{for}~ j\geq1.$$
	If $n=n_{i}$ for some $i$, then for $j=j_{i}$ we have $n=m_{j},~Y_{n}=Z_{j}$ and
	$$\rho (x_{c}',Y_{n})=\rho (x_{c}',Z_{j})=e_{j}=\frac{c}{a_{m_{j+1}}}d_{n}.$$
	Now, let $n_{i}<n<n_{i+1}$ for some $i$ and $j=j_{i}.$ Then
	\begin{equation}
\label{mn}
m_{j}=n_{i}<n\leq n_{i+1}-1=m_{j+1}.
\end{equation}
It leads to the lower bound of $\rho (x_{c}',Y_{n})$ in terms of $d_{n}$:
	\begin{equation}
	\label{20}
	\rho (x_{c}',Y_{n})\geq \rho (x_{c}',Z_{j+1})=e_{j+1}=cd_{n_{i}}\geq cd_{n}.
	\end{equation}
	To obtain the upper bound of $\rho (x_{c}',Y_{n})$ we observe from (\ref{mn}) that
	\begin{center}
	$\rho (x_{c}',Y_{n})\leq \rho (x_{c}',Y_{n_{i}})=\rho (x_{c}',Z_{j})=e_{j}=\frac{c}{a_{m_{j+1}}}d_{n_{i}}.$
	\end{center}
	Since $n_{i}<n_{i+1}-1<n_{i+1}$, then we have
	\begin{center}
		$a_{n_{i+1}-1}^{2} d_{n_{i}}\leq d_{n_{i+1}-1} \leq d_{n}.$
	\end{center}
	Consequently,
	\begin{equation}
	\label{40}
	\rho (x_{c}',Y_{n})\leq \frac{c}{a_{m_{j+1}}}d_{n_{i}}\leq \frac{c}{{a_{n_{i+1}-1}^{2}a_{m_{j+1}}}}d_{n}\leq \frac{c}{a_{n_{i+1}-1}^{3}}d_{n}.
	\end{equation}
	It follows from (\ref{20}) and (\ref{40}) that
	$$cd_n \leq cd_{n_{i}}\leq\rho(x_{c}', Y_n) \leq \frac{c}{a_{n_{i+1}-1}^{3}}d_{n}.$$
	Notice that $a_{n_{i+1}-1}$ only depends on the sequences $\{n_i\}$ and $\{q_{i}\}$. Therefore by taking supremum over $\{q_{i}\}$ and $\{n_i\}$ we proved
\begin{equation}
\label{e1}
cd_n \leq \rho(x_{c}', Y_n) \leq \tilde{a}cd_n.
\end{equation}
Also note that in  \cite{Ak-Pe} it is shown that, for the same sequences $\{d_n\}$ and $\{Y_n\}$ as in Theorem \ref{101}, there is another element $x_c''\in X$ such that
\begin{equation}
\label{e2}
cd_n \leq \rho(x_{c}'', Y_n) \leq 4cd_n.
\end{equation}
Therefore if $\tilde a\le 4$,
$$
cd_n \leq \rho(x_{c}', Y_n)\le \tilde acd_n=\min(4,\tilde{a})cd_n;
$$
if $\tilde a>4$,
$$
cd_n \leq \rho(x_{c}'', Y_n)\le 4cd_n=\min(4,\tilde{a})cd_n.
$$
Thus by taking
$$
x_c=\left\{\begin{array}{ll}
x_c'&~\mbox{if $\tilde a\le 4$};\\
x_c''&~\mbox{if $\tilde a>4$},
\end{array}\right.
$$
we have proven Theorem \ref{101}.
\end{proof}
\begin{remark}
One can observe that  the inequalities in (\ref{main}) is stronger than the inequalities given by Konyagin in \cite{Konyagin}.
\end{remark}


\begin{thebibliography}{99}

\bibitem{Ak-Al} Aksoy A. G, Almira J.,  \textit{On Shapiro's lethargy theorem and some applications}. Jaen J Approx,  2014; 6(1): 87-116.
\bibitem{Ak-Le2} Aksoy A. G, Lewicki G. , \textit{Bernstein lethargy theorem in Fr\'{e}chet spaces}.   J.  Approx. Theory  2016;  209, 58--77.
\bibitem{Ak-Le} Aksoy A. G, Lewicki G.,  \textit{Diagonal operators, $s$-numbers and Bernstein pairs}.  Note Mat 1999; 17: 209-216.
\bibitem{Ak-Pe} Aksoy A. G., Peng Q., \textit{On a theorem of S. N. Bernstein for Banach spaces}, ArXiv: 1605.04592.
\bibitem{Al-To} Almira J. M, del Toro N., \textit{ Some remarks on negative results in approximation theory.} In:  Proceedings of the Fourth International Conference on Functional Analysis and Approximation Theory, Vol. I;  2000; Potenza; Rend. Circ. Mat. Palermo (2) Suppl., 2002, no. 68, part I, pp. 245-256.
\bibitem{Al-Oi} Almira J. M,  Oikhberg T., \textit{Approximation schemes satisfying Shapiro's theorem.} J Approx. Theory 2012; 534-571.
\bibitem{Bernstein} Bernstein S. N.,  \textit{On the inverse problem in the theory of best approximation of continuous functions}. Collected works (in Russian), Izd Akad Nauk, USSR 1954; vol II:  292-294.
       \bibitem{Borodin}  Borodin P. A., \textit { On the existence of an element with given deviations from an expanding system of subspaces}. Math Notes 2006;  80 (5): 621-630 (Translated from Mathematicheskie Zametki).

	\bibitem{Carleman}  Carleman T., \textit{  Sur un th\'eor\`eme de Weierstrass.} Arkiv Mat Astron Fys 1927; 20: 1-5.
	\bibitem{Cheney}  Cheney E. W., \textit  {Introduction to Approximation Theory}. AMS Chelsea Publishing, Amer Mathematical Society, 1966.
	\bibitem{De-Hu} Deutsch F, Hundal H., \textit { A generalization of Tyuriemskih's lethargy theorem  and some applications}. Numer Func Anal Opt 2013; 34(9):  1033-1040.
 \bibitem{Konyagin} Konyagin S. V., \textit{ Deviation of elements of a Banach space from a system of subspaces.} In: Proceedings of the Steklov Institute of Mathematics; 2014; 284 (1), pp 204-207.
	\bibitem{Lang}  Lang S., \textit{ Math Talks for Undergraduates.} Springer, 1999.
	\bibitem{Lew1} Lewicki G., \textit{  A theorem of Bernstein's type for linear projections. }Univ Iagellon Acta Math 1988; 27: 23-27.
	\bibitem{Ple2} Ple\'{s}niak W., \textit{ Quasianalytic functions in the sense of Bernstein.} Dissert Math 1977; 147: 1-70.
\bibitem{Pli} Plichko A., \textit{Rate of decay of the Bernstein numbers}. Zh Mat Fiz Anal Geom 2013; 9 (1): 59-72.
	\bibitem{Riclin} Rivlin T. J., \textit{ An Introduction to the Approximation of Functions. }Waltham, MA: Blaisdell Pub, 1969.
	\bibitem{Sha} Shapiro H. S., \textit{ Some negative theorems of approximation theory.} Michigan Math J 1964; 11: 211-217.
	\bibitem{Sin} Singer I., \textit{ Best Approximation in Normed Linear Spaces by Elements of Linear Subspaces.} Berlin. Springer-Verlag, 1970.
	\bibitem{Timan} Timan A. F., \textit{ The Theory of Approximating Functions of Real Variables.} Oxford. Macmillan, 1994.
\bibitem {Tyu2}Tyuriemskih I. S., \textit{ The $B$-property of Hilbert spaces.}  Uch Zap Kalinin Gos Pedagog Inst 1964; 39:  53-64.
\bibitem{Tyu}  Tyuriemskih I. S., \textit{  On one problem of S. N. Bernstein.}  Scientific Proceedings of Kaliningrad State Pedagog. Inst.; 1967; 52, 123-129.
\end{thebibliography}
\end{document}